\documentclass[12pt]{amsart}

\title[Superstrong cardinals are never indestructible]{Superstrong and other large cardinals are never Laver indestructible}

\author[Bagaria]{Joan Bagaria}
 \address[J. Bagaria]{ICREA and Universitat de Barcelona, 
                      Departament de L\`ogica, Hist\`oria i Filosofia de la Ci\`encia, 
                      Montalegre 6-8, 08001 Barcelona, Catalonia}
 \email{joan.bagaria@icrea.cat}
 \urladdr{http://www.icrea.cat/Web/ScientificStaff/Joan-Bagaria-i-Pigrau-119}

\author[Hamkins]{Joel David Hamkins}
 \address[J. D. Hamkins]
        {Mathematics, Philosophy, Computer Science,
          The Graduate Center of The City University of New York,
          365 Fifth Avenue, New York, NY 10016
          \&
          Mathematics,
          College of Staten Island of CUNY,
          Staten Island, NY 10314}
\email{jhamkins@gc.cuny.edu}
\urladdr{http://jdh.hamkins.org}

\author[Tsaprounis]{Konstantinos Tsaprounis}
 \address[K. Tsaprounis]
         {Universitat de Barcelona, Departament de L\`ogica, Hist\`oria i Filosofia de la Ci\`encia, Montalegre 6-8, 08001 Barcelona, Catalonia}
 \email{kostas.tsap@gmail.com}

\author[Usuba]{Toshimichi Usuba}
 \address[T. Usuba]
         {Organization of Advanced Science and Technology, Kobe University,
         Rokko-dai 1-1, Nada, Kobe, 657-8501 Japan}
 \email{usuba@people.kobe-u.ac.jp}

\thanks{The research work of the first author was partially supported by the Spanish Ministry of Science and Innovation under grant MTM2011-25229 and by the Generalitat de Catalunya (Catalan Government) under grant 2009-SGR-187. The second author's research has been supported in part by Simons Foundation grant 209252, PSC-CUNY grant 64732-00-42, and CUNY Collaborative Incentive grant 80209-06 20, and he is grateful for the support provided to him as a visitor at the University of Barcelona in December, 2012, where much of this research took place in the days immediately following the dissertation defense of the third author, inspired by remarks made at the defense. The fourth author had subsequently proved essentially similar results independently, and so we joined forces, merged the papers and strengthened the results into the present account. Commentary concerning this article can be made at \href{http://jdh.hamkins.org/superstrong-never-indestructible}{jdh.hamkins.org/superstrong-never-indestructible}.}

\usepackage{amssymb}
\usepackage[hidelinks]{hyperref}
\newtheorem{theorem}{Theorem}
\newtheorem{maintheorem}[theorem]{Main Theorem}
\newtheorem{corollary}[theorem]{Corollary}
\newtheorem{sublemma}{Lemma}[theorem]
\newtheorem{lemma}[theorem]{Lemma}

\newtheorem{question}[theorem]{Question}

\newcommand{\QED}{\end{proof}}

\def\proclaim[#1]{{\bf #1}}
\def\BF#1.{{\bf #1.}}

\newcommand{\Vopenka}{Vop\v{e}nka}

\renewcommand{\P}{{\mathbb P}}
\newcommand{\Q}{{\mathbb Q}}

\newcommand{\Qdot}{{\dot\Q}}

\newcommand{\of}{\subseteq}

\newcommand{\set}[1]{\{\,{#1}\,\}}

\newcommand{\elesub}{\prec}

\newcommand{\cof}{\mathop{\rm cof}}

\newcommand{\Add}{\mathop{\rm Add}}

\newcommand{\Coll}{\mathop{\rm Coll}}

\newcommand{\Con}{\mathop{{\rm Con}}}
\newcommand{\image}{\mathbin{\hbox{\tt\char'42}}}

\newcommand{\restrict}{\upharpoonright} 
\newcommand{\satisfies}{\models}

\newcommand{\intersect}{\cap}

\newcommand{\smalllt}{\mathrel{\mathchoice{\raise2pt\hbox{$\scriptstyle<$}}{\raise1pt\hbox{$\scriptstyle<$}}{\raise0pt\hbox{$\scriptscriptstyle<$}}{\scriptscriptstyle<}}}
\newcommand{\smallleq}{\mathrel{\mathchoice{\raise2pt\hbox{$\scriptstyle\leq$}}{\raise1pt\hbox{$\scriptstyle\leq$}}{\raise1pt\hbox{$\scriptscriptstyle\leq$}}{\scriptscriptstyle\leq}}}
\newcommand{\lt}{\smalllt}

\newcommand{\ltkappa}{{{\smalllt}\kappa}}
\newcommand{\leqkappa}{{{\smallleq}\kappa}}

\newcommand{\ltlambda}{{{\smalllt}\lambda}}
\newcommand{\ltgamma}{{{\smalllt}\gamma}}

\newcommand{\leqdelta}{{{\smallleq}\delta}}
\newcommand{\ltdelta}{{{\smalllt}\delta}}

\newcommand{\boolval}[1]{\mathopen{\lbrack\!\lbrack}\,#1\,\mathclose{\rbrack\!\rbrack}}
\def\[#1]{\boolval{#1}}
\newbox\gnBoxA
\newdimen\gnCornerHgt
\setbox\gnBoxA=\hbox{\tiny$\ulcorner$}
\global\gnCornerHgt=\ht\gnBoxA
\newdimen\gnArgHgt
\def\gcode #1{%
\setbox\gnBoxA=\hbox{$#1$}%
\gnArgHgt=\ht\gnBoxA%
\ifnum     \gnArgHgt<\gnCornerHgt \gnArgHgt=0pt%
\else \advance \gnArgHgt by -\gnCornerHgt%
\fi \raise\gnArgHgt\hbox{\tiny$\ulcorner$} \box\gnBoxA %
\raise\gnArgHgt\hbox{\tiny$\urcorner$}}
\newcommand{\UnderTilde}[1]{{\setbox1=\hbox{$#1$}\baselineskip=0pt\vtop{\hbox{$#1$}\hbox to\wd1{\hfil$\sim$\hfil}}}{}}
\newcommand{\Undertilde}[1]{{\setbox1=\hbox{$#1$}\baselineskip=0pt\vtop{\hbox{$#1$}\hbox to\wd1{\hfil$\scriptstyle\sim$\hfil}}}{}}
\newcommand{\undertilde}[1]{{\setbox1=\hbox{$#1$}\baselineskip=0pt\vtop{\hbox{$#1$}\hbox to\wd1{\hfil$\scriptscriptstyle\sim$\hfil}}}{}}
\newcommand{\UnderdTilde}[1]{{\setbox1=\hbox{$#1$}\baselineskip=0pt\vtop{\hbox{$#1$}\hbox to\wd1{\hfil$\approx$\hfil}}}{}}
\newcommand{\Underdtilde}[1]{{\setbox1=\hbox{$#1$}\baselineskip=0pt\vtop{\hbox{$#1$}\hbox to\wd1{\hfil\scriptsize$\approx$\hfil}}}{}}

\newcommand{\st}{\mid}
\renewcommand{\th}{{\hbox{\scriptsize th}}}

\newcommand{\iso}{\cong}
\def\<#1>{\langle#1\rangle}

\newcommand{\Ord}{\mathop{{\rm Ord}}}

\newcommand{\ZFC}{{\rm ZFC}}

\newcommand{\ZC}{{\rm ZC}}
\newcommand{\KM}{{\rm KM}}

\newcommand{\GBC}{{\rm GBC}}

\newcommand{\cell}[1]{\boxit{\hbox to 17pt{\strut\hfil$#1$\hfil}}}
\newcommand{\head}[2]{\lower2pt\vbox{\hbox{\strut\footnotesize\it\hskip3pt#2}\boxit{\cell#1}}}
\newcommand{\boxit}[1]{\setbox4=\hbox{\kern2pt#1\kern2pt}\hbox{\vrule\vbox{\hrule\kern2pt\box4\kern2pt\hrule}\vrule}}
\newcommand{\Col}[3]{\hbox{\vbox{\baselineskip=0pt\parskip=0pt\cell#1\cell#2\cell#3}}}
\newcommand{\tapenames}{\raise 5pt\vbox to .7in{\hbox to .8in{\it\hfill input: \strut}\vfill\hbox to
.8in{\it\hfill scratch: \strut}\vfill\hbox to .8in{\it\hfill output: \strut}}}
\newcommand{\Head}[4]{\lower2pt\vbox{\hbox to25pt{\strut\footnotesize\it\hfill#4\hfill}\boxit{\Col#1#2#3}}}
\newcommand{\Dots}{\raise 5pt\vbox to .7in{\hbox{\ $\cdots$\strut}\vfill\hbox{\ $\cdots$\strut}\vfill\hbox{\
$\cdots$\strut}}}
\newcommand{\df}{\it} 
\begin{document}

\begin{abstract}
Superstrong cardinals are never Laver indestructible. Similarly, almost huge cardinals, huge cardinals, superhuge cardinals, rank-into-rank cardinals, extendible cardinals, $1$-extendible cardinals, $0$-extendible cardinals, weakly superstrong cardinals, uplifting cardinals, pseudo-uplifting cardinals, superstrongly unfoldable cardinals, $\Sigma_n$-reflecting cardinals, $\Sigma_n$-correct cardinals and $\Sigma_n$-extendible cardinals (all for $n\geq 3$) are never Laver indestructible. In fact, all these large cardinal properties are superdestructible: if $\kappa$ exhibits any of them, with corresponding target $\theta$, then in any forcing extension arising from nontrivial strategically $\ltkappa$-closed forcing $\Q\in V_\theta$, the cardinal $\kappa$ will exhibit none of the large cardinal properties with target $\theta$ or larger.
\end{abstract}

\maketitle

\section{Introduction}

\noindent
The large cardinal indestructibility phenomenon, occurring when certain preparatory forcing makes a given large cardinal become necessarily preserved by any subsequent forcing from a large class of forcing notions, is pervasive in the large cardinal hierarchy. The phenomenon arose in Laver's seminal result~\cite{Laver78} that any supercompact cardinal $\kappa$ can be made indestructible by $\ltkappa$-directed closed forcing. It continued with the Gitik-Shelah~\cite{GitikShelah89} treatment of strong cardinals; the universal indestructibility of Apter and Hamkins~\cite{ApterHamkins99:UniversalIndestructibility}, which produced simultaneous indestructibility for all weakly compact, measurable, strongly compact, supercompact cardinals and others;  the lottery preparation of Hamkins~\cite{Hamkins2000:LotteryPreparation}, which applies generally to diverse large cardinals; work of Apter, Gitik and Sargsyan on indestructibility and the large-cardinal identity crises~\cite{ApterGitik98, Apter2006:IndestructibilityAndStrongCompactness, Apter2006:LeastStronglyCompactCanBeLeastStrongAndIndestructible, Sargsyan2009:OnIndestructibilityAspectsOfIdentityCrises}; the indestructibility of strongly unfoldable cardinals~\cite{Johnstone2007:Dissertation, Johnstone2008:StronglyUnfoldableCardinalsMadeIndestructible}; the indestructibility of \Vopenka's principle~\cite{Brooke-Taylor2011:IndestructibilityOfVopenkaPrinciple}; and diverse other treatments of large cardinal indestructibility. Based on these results, one might be tempted to the general conclusion that all the usual large cardinals can be made indestructible. (Meanwhile, results in~\cite{Hamkins94:FragileMeasurability,Hamkins94:Dissertation,Hamkins98:SmallForcing, HamkinsShelah98:Dual, Hamkins99:GapForcingGen, Hamkins2001:GapForcing, Hamkins2003:ExtensionsWithApproximationAndCoverProperties} show the dual result that large cardinal properties can in contrast be made destructible, and furthermore that small forcing quite generally ruins indestructibility.)

In this article, we temper that temptation by proving that certain kinds of large cardinals cannot be made nontrivially indestructible. Superstrong cardinals, we prove, are never Laver indestructible. Consequently, neither are almost huge cardinals, huge cardinals, superhuge cardinals, rank-into-rank cardinals, extendible cardinals and $1$-extendible cardinals, to name a few. Even the $0$-extendible cardinals are never indestructible, and neither are weakly superstrong cardinals, uplifting cardinals, pseudo-uplifting cardinals, strongly uplifting cardinals, superstrongly unfoldable cardinals, $\Sigma_n$-reflecting cardinals, $\Sigma_n$-correct cardinals and $\Sigma_n$-extendible cardinals, when $n\geq 3$. (A cardinal $\kappa$ is $\Sigma_n$-extendible---or more precisely, $(\Sigma_n,0)$-extendible, since it is a weakening of $0$-extendibility---if there is some $\theta>\kappa$ with $V_\kappa\elesub_{\Sigma_n}V_\theta$; see~\S\ref{Section.Background}.) In fact, all these large cardinal properties are superdestructible, in the sense that if $\kappa$ exhibits any of them, with corresponding target $\theta$, then in any forcing extension arising from nontrivial strategically $\ltkappa$-closed forcing $\Q\in V_\theta$, the cardinal $\kappa$ will exhibit none of the large cardinal properties with target $\theta$ or larger. Our strongest result in this line is expressed by main theorem~\ref{Theorem.Sigma2Sigma3Extendible}, asserting that if $\kappa$ is $\Sigma_2$-extendible to target $\theta$ or higher in $V$, then it is not $\Sigma_3$-extendible to target $\theta$ or higher in any nontrivial forcing extension by strategically $\ltkappa$-closed forcing $\Q\in V_\theta$. We can drop the assumption that $\kappa$ has any large cardinal property in $V$ by restricting the class of forcing somewhat, as in theorem~\ref{Theorem.Q=QxQhomogeneous}. Corollary~\ref{Corollary.Add(kappa,1)Etc.} shows as a consequence that many quite ordinary forcing notions, which one might otherwise have expected to fall under the scope of an indestructibility result, will definitely ruin all these large cardinal properties. For example, adding a Cohen subset to any cardinal $\kappa$ will definitely prevent it from being superstrong---as well as preventing it from being uplifting, $\Sigma_3$-correct, $\Sigma_3$-extendible and so on with all the large cardinal properties mentioned above---in the forcing extension.

\begin{maintheorem}\label{Theorem.MainTheorem}\
\begin{enumerate}
 \item Superstrong cardinals are never Laver indestructible.
 \item Consequently, almost huge, huge, superhuge and rank-into-rank cardinals are never Laver indestructible.
 \item Similarly, extendible cardinals, $1$-extendible and even $0$-extendible cardinals are never Laver indestructible.
 \item Uplifting cardinals, pseudo-uplifting cardinals, weakly superstrong cardinals, superstrongly unfoldable cardinals and strongly uplifting cardinals are never Laver indestructible.
 \item $\Sigma_n$-reflecting and indeed $\Sigma_n$-correct cardinals, for each finite $n\geq 3$, are never Laver indestructible.
 \item Indeed---the strongest result here, because it is the weakest notion---$\Sigma_3$-extendible cardinals are never Laver indestructible.
\end{enumerate}
In fact, each of these large cardinal properties is superdestructible. Namely, if $\kappa$ exhibits any of them, with corresponding target $\theta$, then in any forcing extension arising from nontrivial strategically $\ltkappa$-closed forcing $\Q\in V_\theta$, the cardinal $\kappa$ will exhibit none of the mentioned large cardinal properties with target $\theta$ or larger.
\end{maintheorem}

Precise definitions of the large cardinal properties appear at the end of section~\ref{Section.Background}. Statement (5) of main theorem~\ref{Theorem.MainTheorem} is technically a theorem scheme, a separate statement for each finite $n$ in the meta-theory.

We shall prove main theorem~\ref{Theorem.MainTheorem} as a corollary to the following strengthened version of statement (6), which we state here separately since it is the focal case and may be understood without any special terminology.

\begin{maintheorem}\label{Theorem.Sigma2Sigma3Extendible}
Suppose that $V_\kappa\elesub_{\Sigma_2}V_\lambda$ for some $\lambda\geq\eta$ and that $G\of\Q$ is $V$-generic for nontrivial strategically $\ltkappa$-closed forcing $\Q\in V_\eta$. Then for all $\theta\geq\eta$,
$$V_\kappa=V[G]_\kappa\not\elesub_{\Sigma_3}V[G]_\theta.$$
\end{maintheorem}

In other words, if $\kappa$ is $\Sigma_2$-extendible with target above $\eta$ in $V$, then after any nontrivial strategically $\ltkappa$-closed forcing in $V_\eta$, it is not $\Sigma_3$-extendible with target above $\eta$.

Before continuing, we should like to express how very pleased and honored we are to be a part of this special volume in the memory of Richard Laver, whose mathematics has inspired and informed so much of our own work and, indeed, of the entire field. We are especially pleased to participate with this particular article, as it happens to make its progress specifically by building connections between two of Laver's most fundamental contributions, namely, the large-cardinal indestructibility phenomenon he pioneered in \cite{Laver78} and the ground-model definability phenomenon he discovered in \cite{Laver2007:CertainVeryLargeCardinalsNotCreated}. Our main theorem explains how ground-model definability places limits on the extent of indestructibility.

\section{Some geological background and definitions}\label{Section.Background}

Our proof will make use of recent results on the definability of ground models in the emerging topic known as set-theoretic geology, the study of the collection of ground models over which the universe $V$ was obtained by forcing. Theorem~\ref{Theorem.Grounds} summarizes the basic situation, and is closely related to Laver's theorem~\cite{Laver2007:CertainVeryLargeCardinalsNotCreated}, proved independently by Woodin, on the definability of grounds, namely, the fact that if $V\of V[G]$ is a forcing extension via $V$-generic filter $G\of\Q\in V$, then $V$ is a definable class in $V[G]$; it is also closely related to Hamkins's strengthening of that theorem to the pseudo-grounds, namely, if $W\of V$ has the $\delta$-approximation and cover properties (defined below) and $(\delta^+)^W=\delta^+$, then $W$ is definable in $V$ using parameter $r=({}^\ltdelta 2)^W$. A transitive class model $W$ of \ZFC\ is a {\df ground} of $V$ if $V=W[G]$ is a forcing extension via some $W$-generic $G\of\Q\in W$.

\begin{theorem}[\cite{FuchsHamkinsReitz:Set-theoreticGeology}]\label{Theorem.Grounds}
There is a parameterized family $\set{W_r\mid r\in V}$ of
transitive classes such that:
 \begin{enumerate}
  \item Every $W_r$ is a ground of $V$.
  \item Every ground of $V$ is $W_r$ for some $r$.
  \item The family is uniformly definable in that the relation ``$x\in W_r$'' is definable without parameters.
 \end{enumerate}
\end{theorem}

An immediate consequence of theorem~\ref{Theorem.Grounds} is that many second-order-seeming assertions about how the set-theoretic universe $V$ might have been obtained by forcing are actually first-order expressible in the language of set theory. For example, the assertion that $V$ is not a set-forcing extension of any inner model, an assertion known as the {\df ground axiom}, introduced by Hamkins and Reitz~\cite{Reitz2006:Dissertation,Reitz2007:TheGroundAxiom,Hamkins2005:TheGroundAxiom}, is expressed simply as ``$\forall r\, V=W_r$.'' Similarly, it follows from theorem~\ref{Theorem.Grounds} that the assertion that the universe is obtained by forcing of this or that special kind over some ground model is first-order expressible in the language of set theory.

Our argument will rely not only on the statement of theorem~\ref{Theorem.Grounds}, but also on some of the ideas and finer details of the proof, for we aim to give special attention to the complexity of the assertions that the universe was obtained in a particular way by forcing. We shall furthermore want to apply theorem~\ref{Theorem.Grounds} not only in the full set-theoretic universe $V$, but also in $V_\theta$, whenever $\theta$ is a limit of $\beth$-fixed points of arbitrarily large cofinality below $\theta$. (It will suffice, for example, that $\theta$ is a fixed point of the enumeration of $\beth$-fixed points.) So let us now explain some of those ideas and finer details.

The proof makes use of the following central definition of~\cite{Hamkins2003:ExtensionsWithApproximationAndCoverProperties}, concerning extensions $W\of U$ of transitive models of set theory, where $\delta$ is a cardinal in $U$.
\begin{enumerate}
 \item  The extension $W\of U$ satisfies the {\df
     $\delta$-approximation property} if whenever $A\of
     W$ is a set in $U$ and $A\intersect a\in W$
     for any $a\in W$ of size less than $\delta$ in
     $W$, then $A\in W$.
 \item The extension $W\of U$ satisfies the {\df
     $\delta$-cover property} if whenever $A\of W$ is a
     set of size less than $\delta$ in $U$, then
     there is a covering set $B\in W$ with $A\of B$ and
     $|B|^W<\delta$.
\end{enumerate}
The core fact is that every inner model $W\of V$ of \ZFC\ exhibiting the $\delta$-approximation and cover properties and with the right $\delta^+$ is uniquely characterized by those facts and its power set $P(\delta)$. A level-by-level analogue of this is stated in theorem~\ref{Theorem.ApproximationCoverImpliesM=N}. Since lemma~\ref{Lemma.ClosurePoint} shows that every set-forcing extension exhibits the $\delta$-approximation and cover properties for some $\delta$ (note that $\Qdot$ can be trivial there), and since $P(\delta)^W$ is determined via the $\delta$-approximation property by $({}^\ltdelta 2)^W$, it follows that every ground $W$ will be definable from the parameter $r=({}^\ltdelta 2)^W$ for such a $\delta$, and in theorem~\ref{Theorem.Grounds}, we will accordingly have $W=W_r$ for $r=({}^\ltdelta 2)^W$. Jonas Reitz~\cite{Reitz2006:Dissertation,Reitz2007:TheGroundAxiom} isolated the convenient and economical theory $\ZFC_\delta$, making for an easy statement of theorem~\ref{Theorem.ApproximationCoverImpliesM=N}. Specifically, $\ZFC_\delta$ has the axioms of Zermelo set theory, the axiom of choice, the $\leqdelta$-replacement axiom (meaning instances of the replacement axiom for functions with domain $\delta$, a fixed regular cardinal), together with the axiom asserting that every set is coded by a set of ordinals. The theory $\ZFC_\delta$ can be formalized in the language of set theory augmented by a constant symbol for $\delta$, or it can be viewed as making assertions about a particular regular cardinal $\delta$ that has already been fixed. For example, assuming \ZFC\ in the background, then for any regular cardinal $\delta$ and $\beth$-fixed point $\theta$ with cofinality larger than $\delta$, it is an easy exercise to verify that $V_\theta\satisfies\ZFC_\delta$ using that $\delta$.

A forcing notion $\Q$ is {\df $\ltkappa$-closed}, if any descending sequence of conditions in $\Q$ of length less than $\kappa$ has a lower bound in $\Q$. More generally, $\Q$ is {\df strategically} $\ltkappa$-closed, if there is a strategy $\tau$ enabling player II to continue legal play in the game of length $\kappa$, where the players take turns in specifying the next element in a descending sequence $\<p_\alpha\st \alpha<\kappa>$ from $\Q$, with player II going first at limit stages (so player I wins if during play a descending sequence of length less than $\kappa$ is constructed that has no lower bound for player II to play).

\begin{theorem}[{Hamkins, see \cite[lemma 7.2]{Reitz2007:TheGroundAxiom}}]\label{Theorem.ApproximationCoverImpliesM=N}
Suppose that $M$, $N$ and $U$ are transitive models of $\ZFC_\delta$, where $\delta$ is a fixed regular cardinal, that $M\of U$ and $N\of U$ have the $\delta$-approximation and $\delta$-cover properties, and that $P(\delta)^M=P(\delta)^N$ and $(\delta^+)^M=(\delta^+)^N=(\delta^+)^U$. Then~$M=N$.
\end{theorem}

Hamkins and Johnstone observed (2012) that we may easily weaken the assumption that $P(\delta)^M=P(\delta)^N$ to the assumption merely that $({}^{\ltdelta}2)^M=({}^\ltdelta 2)^N$, since under the $\delta$-approximation property these are equivalent: if $A\of\delta$ and $A\in M$, then every initial segment of $A$ is in $({}^\ltdelta 2)^M$ and hence would also be in $N$ under that assumption, making $A\in N$ by the $\delta$-approximation property; and similarly in the other direction for $A\in N$, leading to $P(\delta)^M=P(\delta)^N$. In the case that $\delta=\lambda^+$, it suffices merely that $P(\lambda)^M=P(\lambda)^N$, since $P(\lambda)$ determines ${}^{\ltlambda^+}2$, which as we said ensures $P(\lambda^+)^M=P(\lambda^+)^N$ via the $\delta$-approximation property. We shall henceforth regard these improvements as a part of theorem~\ref{Theorem.ApproximationCoverImpliesM=N}.

\goodbreak
\begin{lemma}[{\cite[lemma
13]{Hamkins2003:ExtensionsWithApproximationAndCoverProperties}, see \cite[lemma 12]{HamkinsJohnstone2010:IndestructibleStrongUnfoldability} for an improved proof, following~\cite{Mitchell2003:ANoteOnHamkinsApproximationLemma}}]\label{Lemma.ClosurePoint}
Suppose that $V[g][G]$ is the forcing extension via $g*G\of\P*\Qdot$, where $\P$ is nontrivial, has cardinality less than a regular cardinal $\delta$ for which $\Qdot$ is forced to be strategically $\ltdelta$-closed. Then the extension $V\of V[g][G]$ satisfies the $\delta$-approximation and $\delta$-cover properties.
\end{lemma}

These results are very near to establishing the ground model definability theorem, as well as theorem~\ref{Theorem.Grounds}. We say that a parameter $r$ {\df succeeds in defining a pseudo-ground}, if there is a regular cardinal $\delta$, such that for every $\beth$-fixed point $\gamma$ of cofinality larger than $\delta$, there is a transitive $M\of V_\gamma$, with the $\delta$-approximation and $\delta$-cover properties, such that $M\satisfies\ZFC_\delta$, $(\delta^+)^M=\delta^+$ and $r=({}^\ltdelta 2)^M$. By theorem~\ref{Theorem.ApproximationCoverImpliesM=N}, this $M$ is unique, and if $r$ succeeds in this way, then we let $U_r$ be the union of all such $M$ as $\gamma$ increases without bound. It follows that $U_r\satisfies\ZFC$ and $U_r\of V$ has the $\delta$-approximation and cover properties, the correct $\delta^+$ and $r=({}^\ltdelta 2)^{U_r}$. Conversely, if $U\of V$ is any {\df pseudo-ground}, meaning that it is a model of \ZFC\ with the $\delta$-approximation and cover properties to $V$ for some regular cardinal $\delta$, with $(\delta^+)^U=\delta^+$, then $U=U_r$ for $r=({}^\ltdelta 2)^U$.

Similarly, we say that a parameter $r$ {\df succeeds in defining a ground}, if there is some poset $\Q$ and filter $G\of\Q$, such that for every $\beth$-fixed point $\gamma$ of cofinality larger than $\delta$, there is a transitive $M\of V_\gamma$ with the $\delta$-approximation and cover properties, such that $M\satisfies\ZFC_\delta$, $(\delta^+)^M=\delta^+$, $r=({}^\ltdelta 2)^M$ and $G$ is $M$-generic for $\Q$ with $M[G]=V_\gamma$. In other words, $r$ defines a ground if it defines a pseudo-ground that is a ground. Again,  by theorem~\ref{Theorem.ApproximationCoverImpliesM=N}, for each such $\gamma$ the set $M$ is unique when it exists, and when $r$ succeeds in this way we denote by $W_r$ the union of all such $M$ as $\gamma$ increases without bound. It follows that $W_r\satisfies\ZFC$ is a ground of $V$ via $V=W_r[G]$, and conversely, every ground of $V$ by set forcing arises as some such $W_r$. This establishes theorem~\ref{Theorem.Grounds}, as well as the ground definability theorem, as consequences of theorem~\ref{Theorem.ApproximationCoverImpliesM=N}. (Note, for convenience, when $r$ does not succeed in these definitions, we may define $U_r=V$ and $W_r=V$, respectively, to arrive at a total indexing of all pseudo-grounds and grounds, as stated in theorem~\ref{Theorem.Grounds}.)

In this article, we should like to apply these definitions not only in $V$, where we have \ZFC\ as a background theory, but also in $V_\theta$, when $\theta$ is merely a limit of $\beth$-fixed points $\gamma$ of cofinality larger than $\delta$. Inside such a model $V_\theta$, theorem~\ref{Theorem.ApproximationCoverImpliesM=N} still implies the uniqueness of $M\of V_\gamma$ for each of those $\gamma$ below $\theta$, and so the definition of $W_r$ still makes sense. Further, if $V=W[G]$ is actually a forcing extension by $W$-generic $G\of\Q\in W$, and $\theta$ is above the rank of $\Q$, then it will see suitable $\beth$-fixed points $\gamma$ of cofinality larger than $\delta=|\Q|^+$, for which $M=W_\gamma\of V_\gamma$ will witness the desired $\delta$-approximation and cover properties, so that $W_\theta=W_r^{V_\theta}$. And conversely, any such $W_r^{V_\theta}$ with $V_\theta=W_r[G]$ for some $W_r^{V_\theta}$-generic $G\of\Q\in W_r^{V_\theta}$ will serve our purpose as a ``ground'' with respect to $V_\theta$ in this weak theory context, even though $W_r^{V_\theta}$ may not satisfy all of \ZFC. Any such $W_r^{V_\theta}$ will at least be a union of a nested tower of $\ZFC_\delta$ models.

We briefly illustrate one of the complexity calculations. Consider for a fixed poset $\Q$ and filter $G\of\Q$ the assertion that parameter $r$ succeeds in defining a ground $W_r$ using $\Q$ and $G$, in other words the assertion, ``the universe is $W_r[G]$, obtained by forcing over $W_r$ with the $W_r$-generic filter $G\of\Q\in W_r$.'' This assertion has complexity $\Pi_2(\Q,G,r)$, because the failure of this assertion is observable inside any sufficiently large $V_\xi$, which will see a $\beth$-fixed point $\gamma<\xi$ of cofinality larger than $\delta$, and larger than the rank of $\Q$, for which there is no $M\of V_\gamma$ satisfying $\ZFC_\delta$ and having the $\delta$-approximation and cover properties and the correct value of $\delta^+$, for which $r=({}^\ltdelta 2)^M$ and $G\of\Q\in M$ is $M$-generic, with $V_\gamma=M[G]$. The point now is that an assertion of the form, ``$\exists \xi$ such that $V_\xi$ satisfies $\psi$,'' for an assertion $\psi$ of any complexity, has complexity $\Sigma_2$, and so our statement has complexity $\Pi_2$ in the parameters $\Q$, $G$ and $r$. Similar such analysis will arise in the proof of main theorem~\ref{Theorem.Sigma2Sigma3Extendible}.

Let us conclude this section by providing definitions of all the large cardinal properties appearing in main theorem~\ref{Theorem.MainTheorem}. These are mostly standard notions. A cardinal $\kappa$ is {\df superstrong} if it is the critical point of an elementary embedding $j:V\to M$ from the set-theoretic universe $V$ to a transitive class $M$, for which $V_{j(\kappa)}\of M$. The cardinal $\kappa$ is {\df almost huge}, if it is the critical point of an elementary embedding $j:V\to M$, for which $M^{{\lt}j(\kappa)}\of M$; it is fully {\df huge}, if also $M^{j(\kappa)}\of M$. In each case, the {\df target} is simply the ordinal $j(\kappa)$. The cardinal $\kappa$ is {\df superhuge} if it is huge with arbitrarily large targets. The cardinal $\kappa$ is a {\df rank-into-rank} cardinal if it is the critical point of an embedding $j:V_\lambda\to V_\lambda$, and here again $j(\kappa)$ is the target  (although $\lambda$, which is strictly larger than $j(\kappa)$, may be even more relevant). A cardinal $\kappa$ is {\df extendible}, if for every $\eta$ it is $\eta$-extendible, meaning that it is the critical point of an elementary embedding $j:V_{\kappa+\eta}\to V_\theta$ for some ordinal $\theta$, and again $j(\kappa)$ is the target. In particular, $\kappa$ is $1$-extendible if there is an elementary embedding $j:V_{\kappa+1}\to V_{\theta+1}$ with critical point $\kappa$; and it is $0$-extendible if it is inaccessible and $V_\kappa\elesub V_\theta$ for some $\theta>\kappa$, called the target.

Somewhat lesser known (see~\cite{HamkinsJohnstone2014:ResurrectionAxiomsAndUpliftingCardinals,HamkinsJohnstone:StronglyUpliftingCardinalsAndBoldfaceResurrection}), an inaccessible cardinal $\kappa$ is {\df uplifting}, if $V_\kappa\elesub V_\theta$ for arbitrarily large inaccessible cardinals $\theta$, and it is {\df pseudo-uplifting} if $V_\kappa\elesub V_\theta$ for arbitrary large ordinals $\theta$, without insisting that $\theta$ is inaccessible. Hamkins and Johnstone define that an inaccessible cardinal $\kappa$ is {\df weakly superstrong} if for every transitive set $M$ of size $\kappa$ with $\kappa\in M$ and $M^\ltkappa\of M$, there is a transitive set $N$ and an elementary embedding $j:M\to N$ with critical point $\kappa$, for which $V_{j(\kappa)}\of N$; and it is {\df weakly almost huge} if for every such $M$ there is such $j:M\to N$ for which $N^{<j(\kappa)}\of N$; and as usual $j(\kappa)$ is referred to as the target. A cardinal $\kappa$ is {\df superstrongly unfoldable} if it is weakly superstrong with arbitrarily large targets, and it is {\df almost hugely unfoldable} if it is weakly almost huge with arbitrarily large targets. Remarkably, these concepts are equivalent (see full details in~\cite{HamkinsJohnstone:StronglyUpliftingCardinalsAndBoldfaceResurrection}), and both are equivalent to $\kappa$ being {\df strongly uplifting}, which means that for every $A\of\kappa$, there are arbitrarily large $\theta$ and $A^*\of\theta$ for which $\<V_\kappa,{\in},A>\elesub\<V_\theta,{\in},A^*>$. One may assume without loss of generality that $\theta$ is inaccessible, weakly compact, totally indescribable or more here and similarly with the targets $j(\kappa)$ of the superstrong and almost huge unfoldability embbeddings, respectively.

For two transitive sets $M$ and $N$, we shall write $M\elesub_n N$ or sometimes $M\elesub_{\Sigma_n} N$ to mean that $\<M,{\in}>$ is a $\Sigma_n$-elementary substructure of $\<N,{\in}>$, meaning that $M\of N$ and they agree on the truth of $\Sigma_n$ assertions having parameters in $M$. An ordinal $\eta$ is {\df $\Sigma_n$-correct} if $V_\eta\elesub_n V$. (This is trivial when $n=0$, so we shall consider the concept only when $n\geq 1$.) The class of all $\Sigma_n$-correct cardinals, denoted $C^{(n)}$, is closed and unbounded in $\Ord$ by an application of the reflection theorem. It is easy to see that every $\Sigma_1$-correct ordinal $\eta$ must in fact be a strong limit cardinal, and indeed a $\beth$-fixed point, as well as a fixed point of the enumeration of $\beth$-fixed points (and consequently a limit of $\beth$-fixed points of arbitrarily large cofinality below $\eta$). The least $\Sigma_n$-correct ordinal has cofinality $\omega$, so these cardinals themselves can be singular. Meanwhile, a cardinal $\kappa$ is {\df $\Sigma_n$-reflecting} if it is regular and $\Sigma_n$-correct (or, equivalently, inaccessible and $\Sigma_n$-correct). An unusual meta-mathematical subtlety of these notions is that we have no uniform-in-$n$ concept of the $\Sigma_n$-correct or $\Sigma_n$-reflecting cardinals, but rather a separate notion for each natural number $n$ in the meta-theory. For example, although we can prove in \ZFC\ the existence of the club $C^{(n)}$ of $\Sigma_n$-correct cardinals for each finite $n$ in the meta-theory, we cannot prove or even express the universal assertion ``$\forall n\,\exists\,\Sigma_n$-correct $\kappa$.''\footnote{To see this, let $M\satisfies\ZFC$ be $\omega$-nonstandard and consider the cut determined by the supremum of the least $\Sigma_n$-correct cardinal, for standard $n$ only. This cut is a first-order elementary substructure of $M$, with the same nonstandard $\omega$, but the collection of $n$ for which it has a $\Sigma_n$-correct cardinal has standard $n$ only, and so does not exist in $M$. Meanwhile, in  \GBC\ models having a satisfaction class for first-order truth, such as in any \KM\ model, we do have a concept of $\Sigma_n$-correct that is uniform in $n$, and every model of \KM, even those that are $\omega$-nonstandard, satisfies ``$\forall n\,\exists\, \Sigma_n$-correct $\kappa$.'' Indeed, \KM\ proves that $V$ is the union of a closed unbounded chain of elementary rank initial segments $V_\kappa\elesub V$, and this notion of fully reflecting is expressible in \KM, using the definable truth predicate for first-order truth.} Similarly, the situation $V_\kappa\elesub V$ occurring when $\kappa$ is {\df fully correct}, that is, $\Sigma_n$-correct for all $n$, is not expressible by a single first-order formula, although one can express the theory $``V_\kappa\elesub V$'' as a scheme in the first-order language of set theory augmented by a constant symbol for $\kappa$, asserting that $\kappa$ is $\Sigma_n$-correct for every $n$. Although some find it surprising, this theory is equiconsistent with \ZFC\ by a simple compactness argument: every finite subset of this scheme is realized in any model of \ZFC\ by the reflection theorem. In particular, from $V_\kappa\elesub V\satisfies\ZFC$ we may not conclude that $V\satisfies\Con(\ZFC)$, since the situation is that $V$ knows only of each axiom of \ZFC\ separately that it holds in $V_\kappa$---the assumption $V_\kappa\elesub V$ is a scheme about each axiom separately---and in general we may not put these together to deduce that $V$ satisfies the assertion ``$V_\kappa\satisfies\ZFC$.''

Let us define next that a cardinal $\kappa$ is {\df $(\Sigma_n, 0)$-extendible}---for readability we shall shorten this to just {\df $\Sigma_n$-extendible} in this paper---with target $\theta$, if $V_\kappa\elesub_n V_\theta$. More generally, $\kappa$ is $(\Sigma_n,\eta)$-extendible if there is a $\Sigma_n$-elementary embedding $j:V_{\kappa+\eta}\to V_\theta$, with critical point $\kappa$, for some ordinal $\theta$; this is a weakening of $\kappa$ being $\eta$-extendible. Note that any $\Sigma_n$-correct cardinal is $\Sigma_n$-extendible with arbitrarily large targets (one might say, therefore, that it is super $\Sigma_n$-extendible). As with $\Sigma_n$-reflection, the concept is trivial when $n=0$, so we consider it only when $n\geq 1$. It is an elementary exercise to see that if $\kappa$ is $\Sigma_1$-extendible, witnessed by $V_\kappa\elesub_1 V_\theta$, then $\kappa$ must be a cardinal, a strong limit cardinal, a $\beth$-fixed point, a fixed point of the enumeration of $\beth$-fixed points and consequently a limit of $\beth$-fixed points of arbitrarily large cofinality below $\kappa$. Note also that if $V_\kappa\elesub_n V_\theta$ and $n\geq 1$, then both of these sets satisfy a robust fragment of \ZFC. For example, in addition to extensionality, foundation, pairing, union, power set and choice, we also have the full separation axiom---making for the full Zermelo theory \ZC---simply because these models have the true power set operation of $V$; we also get the collection axiom for $\Sigma_n$ relations in $V_\kappa$, since $V_\kappa\in V_\theta$ serves as a collecting set in $V_\theta$ for any $\Sigma_n$ property. So these are robust models of set theory, even if they don't necessarily satisfy all of \ZFC. And while the cardinals $\kappa$ and $\theta$ might be singular---so we must be on our guard---the $\Sigma_n$-collection axiom exactly ensures that they are $\Sigma_n$-regular, in the sense that there are no $\Sigma_n$-definable singularizing class functions. The meta-mathematical difficulties that we mentioned earlier with the $\Sigma_n$-correct cardinals do not arise with $\Sigma_n$-extendible cardinals, since we make no reference here to truth in $V$, but rather only truth in $V_\kappa$ and $V_\theta$, where we do have a uniform-in-$n$ account.

Finally, for definiteness we say officially that a large cardinal $\kappa$ is {\df Laver indestructible} for a given large cardinal notion, if it retains that large cardinal property in every $\ltkappa$-directed closed forcing extension. Although this class of forcing notions---$\ltkappa$-directed closed forcing---is the class considered by Laver in his supercompactness indestructibility result~\cite{Laver78}, in fact our results here do not depend much on the directed-closed aspect. In particular, because the main theorems establish a superdestructibility result, for which the large cardinal property is destroyed by any forcing notion from a wide class, the numbered claims of main theorem~\ref{Theorem.MainTheorem} will remain true if one has a modified understanding of Laver indestuctibility, provided only that some of those strategically $\ltkappa$-closed forcing notions remain in the new class. For this reason, it is not actually important for us to be very precise about what we mean by Laver indestructibility, although for definiteness, we have been.

\section{Proving the main theorems}

We are now ready to prove the main theorems, beginning with main theorem~\ref{Theorem.Sigma2Sigma3Extendible} and the case of $\Sigma_3$-extendible cardinals, and then deducing main theorem~\ref{Theorem.MainTheorem} as a corollary. We are unsure whether the assumption that $\kappa$ is $\Sigma_2$-extendible in the ground model can be dropped or weakened, although it can be if one makes further assumptions on the forcing $\Q$, as we explain in section \ref{Section.Improvements}.

\begin{theorem}[Main theorem~\ref{Theorem.Sigma2Sigma3Extendible}]
Suppose that $V_\kappa\elesub_2 V_\lambda$ for some $\lambda\geq\eta$ and that $G\of\Q$ is $V$-generic for nontrivial strategically $\ltkappa$-closed forcing $\Q\in V_\eta$. Then for all $\theta\geq\eta$,
$$V_\kappa=V[G]_\kappa\not\elesub_3 V[G]_\theta.$$
\end{theorem}


\begin{proof}
Suppose that $V_\kappa\elesub_2 V_\lambda$ for some $\lambda\geq\eta$, which is to say, that $\kappa$ is $\Sigma_2$-extendible in $V$ with target $\lambda\geq\eta$ and that $\Q\in V_\eta$ is a nontrivial strategically $\ltkappa$-closed notion of forcing. Suppose toward contradiction that
$V[G]_\kappa\elesub_3 V[G]_\theta$ for some $\theta\geq\eta$, or in other words, that $\kappa$ is $\Sigma_3$-extendible with target $\theta$ in the corresponding forcing extension $V[G]$. Since $\kappa$ is a $\beth$-fixed point, as well as a fixed point of the enumeration of $\beth$-fixed points (and a fixed point of the enumeration of {\it those} cardinals), and furthermore can only $\Sigma_3$-extend to such cardinals, we may assume without loss of generality, by increasing $\eta$ if necessary, that $\eta$ also is a $\beth$-fixed point and fixed point of the enumeration of $\beth$-fixed points; we may furthermore assume that $\cof(\eta)>\kappa$, simply by taking the $(\kappa^+)^\th$ next such fixed point, and this is still below $\lambda$ and $\theta$. Since $\Q$ is strategically $\ltkappa$-closed, it follows that $V[G]_\kappa=V_\kappa$, and so the two extendibility hypotheses amount to
 $$V_\kappa\elesub_2 V_\lambda\qquad\text{ and }\qquad V_\kappa\elesub_3 V[G]_\theta.$$
By elementarity, $\lambda$ and $\theta$ are also $\beth$-fixed points and fixed points of the enumeration of $\beth$-fixed points. In particular, each of these cardinals is a limit of $\beth$-fixed points of arbitrarily large cofinality below $\kappa$. Since $\Q$ is small relative to $\eta$, it follows that $V[G]_\eta=V_\eta[G]$, as well as $V[G]_\lambda=V_\lambda[G]$ and $V[G]_\theta=V_\theta[G]$. In particular, $V[G]_\theta$ is a nontrivial forcing extension of $V_\theta$ via $G\of\Q\in V_\theta$. It follows by theorem~\ref{Theorem.Grounds} applied in $V[G]_\theta$ that $V_\theta=W_r^{V[G]_\theta}$ for some parameter $r$, and so $V[G]_\theta$ satisfies the following assertion,
\begin{quote}
 ``For some parameter $r$ and nontrivial poset $\Q$, the universe is $W_r[G]$ for some $W_r$-generic filter $G\of\Q\in W_r$.''
\end{quote}
We claim that this assertion is $\Sigma_3$-expressible in the models in which we are interested, namely, in $V[G]_\theta$, $V_\lambda$ and $V_\kappa$. As in our example complexity calculation in section~\ref{Section.Background}, we can verify in $V[G]_\theta$ that the universe is $W_r[G]$ by inspecting higher and higher rank initial segments of the universe $V[G]_\theta$. Specifically, the displayed assertion above is equivalent to the assertion, ``for some cardinal $\delta$, parameter $r\of ({}^\ltdelta 2)$, nontrivial poset $\Q$ and filter $G\of\Q$, for every $Z$, if $Z=V_{\gamma+2}$ for some ordinal $\gamma$ (that is, $Z$ is the $V_{\gamma+2}$ of the universe in which we interpret the statement, as in $V[G]_{\gamma+2}$, etc.) and $Z$ thinks that $\gamma$ is a $\beth$-fixed point of cofinality larger than $\delta$ for which $V_\gamma^Z\satisfies\ZFC_\delta$, then $Z$ thinks that there is an $M\of V_\gamma^Z$ satisfying $\ZFC_\delta$ such that $M\of V_\gamma^Z$ has the $\delta$-approximation and cover property, such that $r=({}^\ltdelta 2)^M$ and $(\delta^+)^M=\delta^+$, and such that $Z=M[G]$ is the forcing extension of $M$ via $M$-generic filter $G\of\Q\in M$.'' This assertion has complexity $\Sigma_3$, since the part of the assertion about what $Z$ satisfies has all quantifiers bounded by $Z$, and the assertion that ``$Z=V_{\gamma+2}$ for some ordinal $\gamma$'' is $\Pi_1$ in $Z$, since the important part is to say that $Z$ computes its power sets correctly.

It follows now from $V_\kappa\elesub_3 V[G]_\theta$ that $V_\kappa$ must also satisfy the assertion, and so $V_\kappa=W_{r_0}^{V_\kappa}[G_0]$ for some $r_0\in V_\kappa$ and $V_\kappa$-generic filter $G_0\of\Q_0\in V_\kappa$. In particular, $V_\kappa$ satisfies the assertion ``the universe is obtained by forcing over $W_{r_0}$ via $G_0\of\Q_0$,'' an assertion with complexity $\Pi_2(r_0,G_0,\Q_0)$, as we explained in section~\ref{Section.Background}, by asserting that it holds in all the suitable rank initial segments (this is one quantifier rank less complex because we have fixed the parameters $r_0$, $G_0$ and $\Q_0$, rather than quantifying to get them). Since $V_\kappa\elesub_2 V_\lambda$, it follows that $V_\lambda=W_{r_0}^{V_\lambda}[G_0]$, using the same small parameter $r_0$ and small forcing $G_0\of\Q_0$. By cutting down to $\eta$, we also have $V_\eta=W_{r_0}^{V_\eta}[G_0]$. By the details of the indices for ground models, we may assume that $r_0=({}^\ltdelta 2)^{W_{r_0}^{V_\eta}}$, where $\delta=|\Q_0|^+$ in $V$. Combining this with the forcing $V_\eta\of V_\eta[G]$, we conclude that $V[G]_\eta=W_{r_0}^{V_\eta}[G_0][G]$ is a forcing extension of $W_{r_0}^{V_\eta}$ by a nontrivial forcing notion $\Q_0$ of size less than $\delta$ followed by strategically $\ltkappa$-closed forcing $\Q$. By lemma~\ref{Lemma.ClosurePoint}, it follows that $V[G]_\eta$ has the $\delta$-approximation and cover properties over $W_{r_0}^{V_\eta}$.

Similarly, because $V_\kappa\elesub_2 V[G]_\theta$, we know that $V[G]_\theta$ also thinks that it is obtained by forcing over $W_{r_0}^{V[G]_\theta}$ with $G_0\of\Q_0$, and so $V[G]_\theta=W_{r_0}^{V[G]_\theta}[G_0]$. By again cutting down to $\eta$, we also know $V[G]_\eta=W_{r_0}^{V[G]_\eta}[G_0]$, and furthermore we know again that $r_0=({}^\ltdelta 2)^{W_{r_0}^{V[G]_\eta}}$. Since this is forcing of size less than $\delta$, it follows that $V[G]_\eta$ also has the $\delta$-approximation and cover properties over $W_{r_0}^{V[G]_\eta}$. So the situation is that $W_{r_0}^{V_\eta}$ and $W_{r_0}^{V[G]_\eta}$ are both grounds of $V[G]_\eta$ by forcing with the $\delta$-approximation and cover properties and they have the same binary $\delta$-tree $({}^\ltdelta 2)^{W_{r_0}^{V_\eta}}=r_0=({}^\ltdelta 2)^{W_{r_0}^{V[G]_\eta}}$ and the correct $\delta^+$. It therefore follows by theorem~\ref{Theorem.ApproximationCoverImpliesM=N} that $W_{r_0}^{V_\eta}=W_{r_0}^{V[G]_\eta}$. This immediately implies $$V[G]_\eta=W_{r_0}^{V[G]_\eta}[G_0]= W_{r_0}^{V_\eta}[G_0].$$ Note that $G_0\in V_\eta$, since it is in $V[G]$ and has rank less than $\kappa$ and therefore could not have been added by the strategically $\ltkappa$-closed forcing $\Q$. Since furthermore $W_{r_0}^{V_\eta}\of V_\eta$, we conclude from the displayed equation above that $V[G]_\eta\of V_\eta$. This contradicts the nontriviality of $\Q$, thereby proving the theorem.
\end{proof}

We may now deduce the rest of main theorem~\ref{Theorem.MainTheorem} as a consequence.

\begin{proof}[Proof of main theorem~\ref{Theorem.MainTheorem}] Main theorem~\ref{Theorem.Sigma2Sigma3Extendible} amounts to (a strengthening of) main theorem~\ref{Theorem.MainTheorem} statement (6), and the point now is that the rest of main theorem~\ref{Theorem.MainTheorem} is immediately implied by it, for the simple reason that all the other large cardinal properties mentioned in the main theorem imply $\Sigma_3$-extendibility with the same target.

If $\kappa$ is superstrong, for example, witnessed by superstrongness embedding $j:V\to M$ with target $\theta=j(\kappa)$ and $V_\theta=M_\theta$, then it follows that $V_\kappa\elesub M_{j(\kappa)}=V_\theta$, thereby showing that $\kappa$ is $\Sigma_n$-extendible with target $\theta$ for all $n$. Since every almost huge, huge, superhuge, rank-into-rank, extendible and $1$-extendible cardinal is also superstrong with the same target, the same conclusion applies to these cardinals. Every $0$-extendible cardinal is explicitly $\Sigma_n$-extendible for every $n$, and similarly with the uplifting cardinals and the pseudo-uplifting cardinals. If $\kappa$ is weakly superstrong, then there will be embeddings $j:M\to N$ with critical point $\kappa$ and $V_{j(\kappa)}\of N$, which consequently show that $V_\kappa\elesub V_\theta$ for $\theta=j(\kappa)$, witnessing that $\kappa$ is $0$-extendible and thus $\Sigma_n$-extendible for every $n$. And the same reasoning applies to superstrongly unfoldable cardinals, which are weakly superstrong. We already mentioned earlier that every $\Sigma_n$-correct cardinal is $\Sigma_n$-extendible to any $\theta\in C^{(n)}$, and therefore similarly with every $\Sigma_n$-reflecting cardinal.

So if $\kappa$ has any of the large cardinal properties mentioned in main theorem~\ref{Theorem.MainTheorem}, then it is $\Sigma_3$-extendible with the same target $\theta$, and this is destroyed by the forcing $\Q\in V_\theta$, showing that all the other large cardinal properties are also destroyed for target $\theta$ or higher.
\end{proof}

\section{Improvements, alternative proofs, and questions}\label{Section.Improvements}

Since supercompact cardinals and many other large cardinals can be made Laver indestructible, and these cardinals in particular are necessarily $\Sigma_2$-reflecting and hence also $\Sigma_2$-extendible, it follows (assuming the consistency of those large cardinal notions) that main theorem~\ref{Theorem.Sigma2Sigma3Extendible} cannot be improved from $\Sigma_3$-extendibility to $\Sigma_2$-extendibility. That is, we already know of these situations where $\Sigma_2$-extendible cardinals are Laver indestructible.

Nevertheless, it is conceivable that the main theorems could be improved by weakening the hypotheses that $\kappa$ exhibits the large cardinal property in the ground model $V$. That is, in the main theorems we needed to assume that $\kappa$ was (at least) $\Sigma_2$-extendible in the ground model $V$, in order to know that forcing with $\Q$ would destroy $\Sigma_3$-extendibility. Can we omit this assumption? We're not sure. Perhaps this assumption is simply redundant, for it is conceivable that the $\Sigma_3$-extendibility of $\kappa$ in $V[G]$ might imply the required $\Sigma_2$-extendibility of $\kappa$ in the ground model $V$.

\begin{question}
 If $\kappa$ is $\Sigma_3$-extendible with target $\theta$ in a forcing extension $V[G]$ obtained via strategically $\ltkappa$-closed forcing $G\of\Q\in V_\theta$, then must $\kappa$ be $\Sigma_2$-extendible with a target above the rank of $\Q$ in $V$?
\end{question}

If so, then the main theorems could be improved by dropping the assumption on $\kappa$ in the ground model $V$, making instead the plain assertion that after any nontrivial strategically $\ltkappa$-closed forcing $\Q\in V_\theta$, the cardinal $\kappa$ is no longer $\Sigma_3$-extendible with target $\theta$ or higher, and so similarly neither is it superstrong, extendible, almost huge, uplifting, and so on, with such a target. Thus, we would be able to make the same conclusion of the main theorems, while assuming nothing about $\kappa$ in the ground model $V$.

The next theorem shows that in many instances, for particular forcing notions $\Q$, we are able to omit the hypothesis that $\kappa$ is $\Sigma_2$-extendible in the ground model.

\begin{theorem}\label{Theorem.Q=QxQhomogeneous}
 Suppose that $\Q\in V_\theta$ is almost-homogeneous nontrivial strategically $\ltkappa$-closed forcing and that $\Q\cong\Q\times\Q$. Then forcing with $\Q$ destroys the $\Sigma_3$-extendibility of $\kappa$ with target $\theta$ or higher. In particular, after any such forcing $\Q$, the cardinal $\kappa$ is not superstrong, extendible, almost huge, uplifting, pseudo-uplifting and so on with target $\theta$ or higher.
\end{theorem}

\begin{proof}
Suppose that $V[G]$ is a forcing extension by such a $\Q$, and that $\kappa$ is $\Sigma_3$-extendible to $\theta$ in $V[G]$. Since $\Q\cong\Q\times\Q$, we may view $V[G]=V[G_0][G_1]$ as a two-step forcing extension, using $\Q$ each time, where $G\cong G_0\times G_1$. Since we assumed that $\Q$ was almost homogeneous, it follows that all forcing extensions of $V$ by $\Q$ have the same theory about ground model objects, and in particular, $\kappa$ must be $\Sigma_3$-extendible in $V[G_0]$. Since $\Q$ remains strategically $\ltkappa$-closed in $V[G_0]$, we may apply main theorem~\ref{Theorem.Sigma2Sigma3Extendible} to the extension $V[G_0]\of V[G_0][G_1]$ to see that $\kappa$ cannot be $\Sigma_3$-extendible to $\theta$ in $V[G]$ after all, a contradiction.
\end{proof}

\begin{corollary}\label{Corollary.Add(kappa,1)Etc.}
 After forcing with any of the following forcing notions,
 $$\Add(\kappa,1),\quad \Add(\kappa,\kappa^{++}),\quad \Add(\kappa^+,1),\quad \Coll(\kappa^{++},\kappa^{(+)^{\omega^2+5}})$$
 or with any of many other similar forcing notions, a regular cardinal $\kappa$ is not $\Sigma_3$-extendible in the forcing extension. Consequently, $\kappa$ is also not superstrong, extendible, $1$-extendible, $0$-extendible, almost huge, huge, uplifting, pseudo-uplifting, or $\Sigma_3$-correct, etc. in the extension.
\end{corollary}

\begin{proof}
All these forcing notions satisfy the hypothesis of theorem~\ref{Theorem.Q=QxQhomogeneous}. Note that any $\Sigma_3$-extendibility target for $\kappa$ must be above these particular forcing notions.
\end{proof}

We believe that much of the phenomenon of the main theorems is explained by an affirmative answer to the following question, asked by the second author on MathOverflow (and there are many similar questions arising generally with other forcing notions). Here, by adding a Cohen subset to $\kappa$ over a ground model $M$, we mean to force over $M$ with $\Add(\kappa,1)^M$, whose conditions are bounded subsets of $\kappa$ ordered by end-extension.

\begin{question}[\cite{MO120546Hamkins:CanAModelBeACohenExtensionInTwoDifferentWays?}]\label{Question.M[G]=N[H]}
 Does a regular cardinal $\kappa$ necessarily become definable after forcing to add a Cohen subset to it? In particular, if $M$ and $N$ are models of \ZFC\ having a common forcing extension $M[G]=N[H]$, where $G$ is $M$-generic for $\Add(\kappa,1)^M$ and $H$ is $N$-generic for $\Add(\gamma,1)^N$, then must $\kappa=\gamma$?
\end{question}

Theorem~\ref{Theorem.SmallestOrSecond} provides an affirmative answer to the first part of this question, as well as to the second part in the case of inaccessible cardinals. In general, we show that there are at most two regular cardinals of the type mentioned in the question, which nearly answers the question, although we do not know if the case of two cardinals can actually occur; so this part of question~\ref{Question.M[G]=N[H]} remains open. These facts, however, can be used to provide an alternative proof of the main claims of main theorem~\ref{Theorem.MainTheorem}, as we explain after the proof of theorem~\ref{Theorem.SmallestOrSecond}.

Let $\mathcal{C}(\kappa)$ assert that $\kappa$ is a regular cardinal and the universe was obtained by forcing over some ground $W$ to add a Cohen subset to $\kappa$, that is, ``$V=W[G]$ for some ground $W$ and some $W$-generic $G\of\Add(\kappa,1)^W$.''

\begin{theorem}\label{Theorem.SmallestOrSecond} If $\mathcal{C}(\gamma)$ and $\mathcal{C}(\kappa)$ hold, where $\gamma<\kappa$, then $2^\gamma=\kappa$. Consequently,
\begin{enumerate}
 \item There are at most two regular cardinals satisfying property $\mathcal{C}$.
 \item There is at most one inaccessible cardinal satisfying property $\mathcal{C}$.
 \item If $\mathcal{C}(\kappa)$ holds, then $\kappa$ is
 $\Delta_3$-definable either as
    $$\text{``the smallest regular cardinal with property $\mathcal{C}$,''}$$
 or as
    $$\text{``the largest regular cardinal with property $\mathcal{C}$.''}$$
 \item If $\mathcal{C}(\kappa)$ holds and $\kappa$ is inaccessible, then $\kappa$ is $\Pi_2$-definable as
  $$\text{``the inaccessible cardinal with property $\mathcal{C}$.''}$$
\end{enumerate}
Furthermore, these definitions work also in $V_\theta$, whenever $\theta$ is a $\beth$-fixed point of cofinality larger than $2^{2^\kappa}$ or for which $V_\theta$ satisfies $\Pi_2$-collection.
\end{theorem}

In other words, if you force over $V$ to add a Cohen subset $G$ to $\kappa$, then $\kappa$ becomes definable in the specified manner in the forcing extension $V[G]$, and this definition works there in the corresponding $V[G]_\theta$. The proof will proceed via two lemmas.

\begin{sublemma}\label{Lemma.M[G]=N[H]Implies2^QgeqKappa}
 Suppose that $\kappa$ is a regular cardinal, that $M$ and $N$ are transitive class models of \ZFC, and that $V = M[G] = N[H]$ for some $M$-generic $G\of \Q\in M$ and $N$-generic $H\of\P\in N$, where $\Q$ is nontrivial, almost homogeneous, strategically $\ltkappa$-closed and $\Q\cong\Q\times\Q$ in $M$ and $\P$ is nontrivial. Then $\kappa\leq (2^{|\P|})^N$.
\end{sublemma}

\begin{proof}
Let $\gamma=|\P|^N$ and suppose toward contradiction that $(2^\gamma)^N<\kappa$. Assume $\P$ has domain $\gamma$. Let $\delta=\gamma^+$, which has the same meaning in $N$, $M$ and $V$. Force over $V$ with $\Q$ again, to add $V$-generic $G_1\of\Q$, and form the extension $V[G_1]=N[H][G_1]$, which is a forcing extension of $N$ by $\P$ followed by $\Q$. Since $\gamma<\kappa$ and $\Q$ is strategically $\ltkappa$-closed in $N[H]$, it follows by lemma~\ref{Lemma.ClosurePoint}, choosing a suitable $\P$-name for $\Q$ in $N$, that $N\of N[H][G_1]$ has the $\delta$-approximation and cover properties. Note also that $\delta^+$ is the same in all the models mentioned here. Thus, $N=W_r^{V[G_1]}$, where $r=({}^{\ltdelta}2)^N$, and so $V[G_1]=N[H][G_1]$ satisfies the assertion,
\begin{quote}
 ``The universe is a forcing extension of the ground $W_r$ defined by parameter $r$, using forcing $\P$ followed by some further nontrivial strategically $\ltkappa$-closed forcing.''
\end{quote}
The parameters of this assertion are $r$, $\P$ and $\kappa$, which we claim are all in $M$: of course $\kappa$ is in $M$, and $\P\in M[G]$ is small, so it cannot have been added by $G$, and so $\P\in M$; for $r$, observe that because $r\in N\of M[G]$ and $r\of M$ by the closure of $\Q$, it follows from $|r|^N=(2^\gamma)^N<\kappa$ and the closure of $\Q$ again that $r\in M$. Since $V[G_1]=M[G][G_1]$ and $\Q\cong\Q\times\Q$ in $M$, the two-step generic filter $G*G_1$ is isomorphic to a single $M$-generic filter over $\Q$. Since $\Q$ is almost homogeneous, it follows that every forcing extension of $M$ via $\Q$ must satisfy all the same assertions about parameters in $M$. In particular, the displayed assertion above must also be true in $M[G]$ itself, that is to say, in $V$. So $V=M[G]=W_r^V[H_0][G_0]$ for some $W_r^V$-generic $H_0\of\P\in W_r^V$ and some $W_r^V[H_0]$-generic $G_0\of\Q_0\in W_r^V[H_0]$, where $\Q_0$ is nontrivial and strategically $\ltkappa$-closed there. Furthermore, $W_r^V\of M[G]$ has the $\delta$-approximation and cover properties, the correct $\delta^+$, and $r=({}^\ltdelta 2)^{W_r^V}$. Since $N\of N[H]=M[G]$ also has the $\delta$-approximation and cover properties, the correct $\delta^+$, and $r=({}^\ltdelta 2)^N$, it follows by theorem~\ref{Theorem.ApproximationCoverImpliesM=N} that $W_r^V=N$. Note that $H\in N[H]=M[G]=W_r^V[H_0][G_0]$, but it could not have been added by $G_0$, and so $H\in W_r^V[H_0]=N[H_0]$. Thus, $N[H]\of N[H_0]$ and consequently $M[G]\of N[H_0]$, which means in particular that $G_0\in N[H_0]$, contradicting that $\Q_0$ was nontrivial.
\end{proof}

\begin{sublemma}\label{Lemma.2^gamma=kappa}
 Suppose that $\gamma\leq\kappa$ are regular cardinals, that $M$ and $N$ are transitive class models of \ZFC, and that $V = M[G] = N[H]$ for some $M$-generic $G\of \Add(\kappa,1)^M$ and some $N$-generic $H\of\Add(\gamma,1)^N$. Then either $\gamma=\kappa$ or $2^\gamma=\kappa$.
\end{sublemma}

\begin{proof}
Assume $\gamma<\kappa$. Since $\Add(\gamma,1)$ forces $2^\ltgamma=\gamma$ in $N[H]$, it follows that $2^\ltgamma=\gamma$ in $M$, and similarly $2^\ltkappa=\kappa$ in $M[G]$. Both $\Add(\gamma,1)^N$ and $\Add(\kappa,1)^M$ are nontrivial, and the latter is $\ltkappa$-closed in $M$. Thus, using lemma~\ref{Lemma.M[G]=N[H]Implies2^QgeqKappa} for the first step, we observe that $$\kappa\leq (2^{|\Add(\gamma,1)|})^N\leq(2^\gamma)^N=2^\gamma\leq 2^\ltkappa=\kappa,$$
and so $2^\gamma=\kappa$, as desired.
\end{proof}

\begin{proof}[Proof of theorem~\ref{Theorem.SmallestOrSecond}]
It follows directly from lemma~\ref{Lemma.2^gamma=kappa} that if $\mathcal{C}(\gamma)$ and $\mathcal{C}(\kappa)$ and $\gamma<\kappa$, then $2^\gamma=\kappa$. So there cannot be three regular cardinals satisfying $\mathcal{C}$, since any larger instance is determined in this way by the smallest instance, and consequently any regular cardinal satisfying $\mathcal{C}$ will either be the first or the second regular cardinal satisfying $\mathcal{C}$. Thus, any such cardinal is definable by this description in the forcing extension. Similarly, since $2^\gamma$ is never inaccessible, there cannot be two inaccessible cardinals with property $\mathcal{C}$.

Let us now analyze the complexity of these definitions. First, one may easily see that the assertion $\mathcal{C}(\kappa)$ has complexity at most $\Sigma_3$ in the extension, since it is equivalent to the assertion that ``there is some parameter $r$ and some $G$, such that in any $V_{\lambda+2}$ which sees that $\lambda$ is a $\beth$-fixed point of cofinality larger than $\kappa$, then $G\of\Add(\kappa,1)^{W_r^{V_\lambda}}$ is $W_r^{V_\lambda}$-generic and $V_\lambda=W_r^{V_\lambda}[G]$.'' But actually, we can improve this, and we claim that $\mathcal{C}(\kappa)$ has complexity $\Pi_2$. Specifically, we claim that $\mathcal{C}(\kappa)$ in $V$ is equivalent to the statement ``for every structure of the form $V_{\lambda+2}$, where $\lambda$ is a $\beth$-fixed point of cofinality above $\kappa$, there is $r,G\in V_\lambda$ such that $G\of\Add(\kappa,1)^{W_r^{V_\lambda}}$ is $W_r$-generic and $V_\lambda=W_r^{V_\lambda}[G]$.'' This statement has complexity $\Pi_2$, since we are essentially saying ``$\forall Z$, if $Z=V_{\lambda+2}$ for some $\lambda$ and\ldots,'' where the rest of the assertion is all taking place inside $Z$ with all quantifiers bounded by $Z$, and ``$Z=V_{\lambda+2}$ for some $\lambda$'' has complexity $\Pi_1$, since one can say that $Z$ is transitive and satisfies a certain theory and contains all subsets of its members. The point is that even though the assertion allows different $\lambda$ to make use of different $r$ and $G$, it follows from the fact that $r$ and $G$ are bounded in size---we have $r\of {}^{\ltkappa^+}2$ and $G\of H_\kappa$---that there must be some particular $r$ and $G$ that are repeated unboundedly often with increasingly large $\lambda$. Any such unboundedly occurring pair $r$ and $G$ will ensure $V=W_r[G]$, as desired.

Since $\mathcal{C}(\kappa)$ is $\Pi_2$ expressible, it follows easily that the assertions ``$\kappa$ is least with property $\mathcal{C}$'' and ``$\kappa$ is second with property $\mathcal{C}$'' have complexity at worst $\Delta_3$, and the assertion ``$\kappa$ is inaccessible and $\mathcal{C}(\kappa)$'' has complexity $\Pi_2$. Furthermore, these definitions work inside any $V_\theta$, provided that $\theta$ is a $\beth$-fixed point limit of cofinality at least $2^{2^\kappa}$, which is the number of possible $r$'s and $G$'s that might arise. And indeed, one may use $\Pi_2$-collection in place of this cofinality argument to ensure that some $r$ and $G$ are used with unboundedly many $\lambda$ below $\theta$.
\end{proof}

One may use theorem~\ref{Theorem.SmallestOrSecond} to provide an alternative proof of the main non-indestructibility claims of the main theorems. The idea is that after forcing with $\Add(\kappa,1)$, a regular cardinal $\kappa$ cannot be $\Sigma_3$-extendible in the forcing extension, witnessed by $V_\kappa\elesub_3 V[G]_\theta$, since $V[G]_\theta$ would satisfy $\Pi_2$-collection and so $\kappa$ is $\Sigma_3$-definable in $V[G]_\theta$, which is impossible, as the existence of such a cardinal would have to reflect below $\kappa$ in light of $V_\kappa\elesub_3 V[G]_\theta$. Consequently, after forcing to add a Cohen subset to $\kappa$, the cardinal $\kappa$ cannot be $\Sigma_3$-reflecting, superstrong, extendible, uplifting and so on with all the other large cardinal notions we have mentioned in the main theorem. And the same argument works with many other forcing notions.

Let us conclude the paper by pointing out that one may not strengthen the superdestructibility claim of the main theorem from strategically $\ltkappa$-closed forcing to $(\kappa,\infty)$-distributive forcing (and thanks to the referee for this suggestion). Suppose, for example, that $\kappa$ is superstrong in $V$ with target $\theta$, and for simplicity, let us suppose also that $\theta$ itself is inaccessible. Let $\P=\Pi_{\delta<\theta}\Add(\delta^+,1)$ be the Easton-support product up to $\theta$ of the partial orders $\Add(\delta^+,1)$ to add a Cohen subset to $\delta^+$, whenever $\delta<\theta$ is inaccessible. If $G\of\P$ is $V$-generic, then the standard arguments show that $\kappa$ remains superstrong in $V[G]$. Namely, fix any superstrongness extender embedding $j:V\to M$, and then lift the embedding through the forcing $G_\kappa\of\P\restrict\kappa$ up to stage $\kappa$, resulting in $j:V[G_\kappa]\to M[j(G_\kappa)]$, using $j(G_\kappa)=G$; next, use the fact that the rest of the forcing $G_{\kappa,\theta}$ at coordinates in the interval $[\kappa,\theta)$ is $\leqkappa$-distributive over $V[G_\kappa]$, and so being an extender embedding, $j$ lifts uniquely to $j:V[G_\kappa][G_{\kappa,\theta}]\to M[j(G_\kappa)][j(G_{\kappa,\theta})]$, where $j(G_{\kappa,\theta})$ is the filter generated by $j\image G_{\kappa,\theta}$. The lifted embedding $j:V[G]\to M[j(G)]$ has $V[G]_{j(\kappa)}\of M[j(G)]$, because we used the generic filter $G$ below $\theta$, and so $\kappa$ remains superstrong in $V[G]$. Now, let $g\of\kappa^+$ be $V[G]$-generic for $\Add(\kappa^+,1)^V$, that is, adding a second generic filter for the forcing at coordinate $\kappa$. Since adding two subsets of $\kappa^+$ is isomorphic to adding only one, we see that $\kappa$ is superstrong in $V[G][g]$, since this can be viewed as $V[G^*]$, where $G^*$ is the same as $G$, except that $G^*(\kappa)\iso G(\kappa)*g$ at coordinate $\kappa$ incorporates the two generic filters together at that stage into one. So the situation here is that a superstrong cardinal $\kappa$ in $V[G]$ remains superstrong after the forcing $\Add(\kappa^+,1)^V$ over $V[G]$, even though this forcing is $(\kappa,\infty)$-distributive in $V[G]$, which is the residue in $V[G]$ of the fact that it is $\leqkappa$-closed in $V$. So we cannot expect to prove that superstrongness is always destroyed by such forcing.

\bibliographystyle{alpha}
\bibliography{MathBiblio,HamkinsBiblio}
\end{document}